\documentclass[11pt]{amsart}

\usepackage{amssymb,amsfonts,mathtools}
\usepackage[all,arc]{xy}
\usepackage{enumerate}
\usepackage{mathrsfs}
\usepackage{amsmath}
\usepackage{hyperref}
\usepackage{amsthm}
\usepackage{tikz}
\usepackage{tikz-cd}
\usepackage[utf8]{inputenc}
\usepackage{lipsum}
\usepackage{extpfeil}
\usepackage{colonequals}
\usepackage{graphicx}
\usepackage{subfig}
\usepackage{enumitem}
\usepackage{array}

\newtheorem{thm}{Theorem}[section]

\newtheorem{prop}[thm]{Proposition}
\newtheorem{lem}[thm]{Lemma}

\theoremstyle{definition}
\newtheorem{defn}[thm]{Definition}

\newtheorem*{exmp}{Example}
\newtheorem*{exmps}{Examples}
\newtheorem{alg}{Algorithm}

\theoremstyle{remark}
\newtheorem*{rmk}{Remark}
\newtheorem*{rmks}{Remarks}

\newcommand{\binomd}[2]{\left(\begin{matrix}#1\\ #2 \end{matrix}\right)}

\DeclareMathOperator{\ZZ}{\mathbb{Z}}

\DeclareMathOperator{\NN}{\mathbb{N}}

\DeclareMathOperator{\FF}{\mathbb{F}}

\DeclareMathOperator{\SL}{\mathrm{SL}}

\newcommand{\card}[1]{\# #1}

\usepackage{float}
\usepackage[margin=1in]{geometry}
\usepackage{array}
\usepackage{comment}
\usepackage[symbol]{footmisc}
\newcommand{\Mod}[1]{\ (\mathrm{mod}\ #1)}

\subjclass[2010]{ <>; <>}

\title[Hecke nilpotency for modular forms mod 2 and partition numbers]{Hecke nilpotency for modular forms mod 2 and an application to partition numbers}

\author[C. Cossaboom]{Catherine Cossaboom}
\address{Department of Mathematics, University of Virginia, 141 Cabell Drive, Charlottesville, VA 22903}
\email{qkb9us@virginia.edu}

\author[S. Zhou]{Sharon Zhou}

\address{Department of Mathematics, University of Chicago, 5734 South University Avenue, Chicago, IL 60637} 

\email{zhous@uchicago.edu} 

\date{}



\begin{document}

\maketitle
\begin{abstract}
A well-known observation of Serre and Tate is that the Hecke algebra acts locally nilpotently on modular forms mod 2 on $\SL_2(\ZZ)$. We give an algorithm for calculating the degree of Hecke nilpotency for cusp forms, and we obtain a formula for the total number of cusp forms mod 2 of any given degree of nilpotency. Using these results, we find that the degrees of Hecke nilpotency in spaces $M_k$ have no limiting distribution as $k \rightarrow \infty$.  As an application, we study the parity of the partition function using Hecke nilpotency.
\end{abstract}

\section{Introduction and Statement of Results}

The theory of integer weight modular forms \cite{CBMS} is ubiquitous throughout modern number theory, playing a key role in the study of elliptic curves, quadratic forms, and partition functions, to name a few.
The Delta function, often defined as the infinite product
\[
\Delta(q)\coloneqq q\prod_{n=1}^\infty(1-q^n)^{24}=:\sum_{n=1}^\infty \tau(n) q^n \quad   (q\coloneqq e^{2\pi i z}),
\]
is a prototypical example of modular forms. It is a weight 12 cusp form on the full modular group $\SL_2(\ZZ)$ and appears extensively in modern number theory. The coefficients of $\Delta$ are called Ramanujan's $\tau$-function.
Based on computational evidence, Ramanujan \cite{Ramanujan} conjectured that,
    for any prime $p$, one has
    \[
    |\tau (p)|\le 2p ^{11/2}.
    \]
This assertion was proved by Deligne as part of his work on the Weil conjectures \cite{DeligneWeil}. The famous Ramanujan-Petersson Conjecture, which generalizes this statement to cusp forms of all weights, was proved by Deligne \cite{Delignecusp}, Eichler-Shimura \cite{Eichler, Shimura}, and Deligne-Serre \cite{DelSer}.

Another observation of Ramanujan is that the $\tau$-function satisfies various congruences modulo small prime powers, which relate the $\tau$-function to certain divisor functions. For example, we have
\begin{align*}
    \tau(n)&\equiv \sigma_{11}(n) \pmod{2^6},\\
    \tau(n)&\equiv n^2\sigma_3(n) \pmod{7},\\
    \tau(n)&\equiv \sigma_{11}(n) \pmod{691}.
\end{align*} 
These congruences inspired Serre and Swinnerton-Dyer to develop the theory of congruences of modular forms modulo primes, which later played a key role in the early development of the theory of modular Galois representations \cite{Swinn-Dyer}. 

Let $M_k$ denote the space of weight $k$ $(k\ge 4, \text{ even})$ modular forms over $\SL_2(\ZZ)$, and recall that $f\in M_k$ can be identified by its Fourier expansion $f(z)= \sum_{n=0}^\infty a(n) q^n$. Let $M_k(\ZZ)\coloneqq M_k\cap \ZZ[[q]]$ denote the subset of weight $k$ modular forms with integer Fourier coefficients. Elements in $M_k(\ZZ)$ are isobaric polynomials in $E_4$ and $E_6$, where $E_k$ denotes the classical (normalized) Eisenstein series of weight $k$. 

Given any $f\in M_k(\ZZ)$ and a prime number $\ell$, one obtains a $q$-series $\Tilde{f}$ with coefficients in $\FF_\ell$ by reducing the coefficients $a(n)$ modulo $\ell$. For example, as a consequence of the congruence of $\tau(n)$ modulo $2^6$, we have
\begin{align}\label{oddsq}
\Delta\equiv \sum_{m=1}^\infty q^{(2m+1)^2}  \pmod{2}.
\end{align}
The graded ring of modular forms mod $\ell$ is defined as the ring of all $q$-series obtained in this way. In other words, it is the graded ring over $\FF_\ell$ generated by $\{ \Tilde{f}: f\in M_k(\ZZ), k\ge 4 \}$.  The structure of these graded rings was determined by Swinnerton-Dyer \cite{Swinn-Dyer}. 

For $\ell\ge 5$, the graded ring of modular forms mod $\ell$ can be identified as the quotient ring $\FF_\ell[E_4,  E_6]/(\Bar{A}-1)$, where $\Bar{A}$ is the reduction mod $\ell$ of the weight $\ell-1$ isobaric polynomial $A$ with $\ell$-integral coefficients such that $A(E_4, E_6) = E_{\ell - 1}$ (see Theorem 2 in \cite{Swinn-Dyer}). For $\ell = 2$ or 3, the situation is simpler: since 
\begin{align}\label{eisenstein}
E_4(z)=1+ 240 \sum_{n\ge 1} \sigma_3(n) q^n,\quad  E_6(z)=1-504 \sum_{n\ge 1} \sigma_5(n) q^n,
\end{align}
we have $E_4\equiv E_6\equiv1\pmod 6$. Thus,  the graded ring of modular forms mod $\ell$ is isomorphic to $\FF_\ell[\Delta]$ (see Theorem 3 in \cite{Swinn-Dyer}).

A key tool in the study of modular forms is the algebra of Hecke operators acting on $M_k$. For primes $p$, the Hecke operator $T_p$ acts on $f(q)=\sum_{n=0}^\infty a(n) q^n$ by
\begin{align*}
T_p\mid f(q)\coloneqq\sum_{n=1}^\infty \left (a(pn) q^n+p^{k-1} a(n/p)\right) q^n,
\end{align*}
where $a(n/p) \coloneqq 0$ if $p\nmid n$. Clearly, for modular forms mod 2, the action of $T_p$ depends only on the Fourier coefficients and not the weight of the modular form. Thus, the action of the Hecke operators descends naturally to the graded ring $\FF_2[\Delta]$ of modular forms mod 2 in the following way:
\begin{align}\label{heckeoperator}
    &T_p\mid f(q)\equiv \sum_{n=1}^\infty (a(pn)+a(n/p)) q^n \pmod 2 \quad (p \ne 2),\\
    &T_2\mid f(q)\equiv \sum_{n=1}^\infty a(2n) q^n \pmod 2. \label{heckeoperator2}
\end{align}

Serre \cite{Serre} conjectured and Tate \cite{Tate} proved that the Hecke operators $T_p$ for odd primes $p$ act locally nilpotently on $\FF_2[\Delta]$. 
More precisely, given any $f \in \FF_2[\Delta]$ with $f \not\equiv 0\pmod 2$, there is a well-defined integer $d(f)$, called the \emph{degree of nilpotency} of $f$, defined as the smallest integer $d$ such that
\[
T_{p_1} T_{p_2} \dots T_{p_d} \mid f \equiv 0 \pmod{2}
\]
for any collection of $d$ odd primes $p_1, \ldots, p_d$. If $f \equiv 0 \pmod 2$, we set $d(f)\coloneqq-\infty.$ For example, we have $d(\Delta)=1$, $d(\Delta^9)=3$, $d(\Delta^{17})=3$. In fact, $\Delta$ is the only modular form with degree of nilpotency 1 (see Proposition 2.37 in \cite{CBMS}).

Incorporating the operator $T_2$ into the theory requires a little more care. By \eqref{eisenstein}, we see that $T_2\mid E_4\equiv 1\pmod 2$, so $E_4$ is never annihilated by iterating $T_2$. Thus, in order to include $T_2$ in the definition of the degree of nilpotency, we must restrict to cusp forms. For a precise statement, see Definition \ref{newdeg}. 

One may naturally wonder how the degree of nilpotency may be computed. In general, it is \emph{prima facie} a difficult task to compute $d(f)$ for an arbitrary $f\in\FF_2[\Delta]$. For example, if $f=\Delta^{128}+\Delta^{60}+\Delta^7+\Delta^4$, it can be quite laborious to compute the image of $f$ under various $T_p$ and determine the minimum size of families of Hecke operators which always annihilate $f$. Another difficulty arises from the fact that, while we have the trivial upper bound $d(f+g)\le \max\{d(f), d(g)\}$, equality does not always hold. For example, if $f=\Delta^3$ and  $g=\Delta^3+\Delta$, then $d(f)=d(g)=2$, but $d(f+g)=d(\Delta)=1$. In fact, a key step in computing $d(f)$ is determining the condition under which this equality is attained.

In this paper, we describe a method for computing the degree of nilpotency for cusp forms mod 2.
The setup is as follows: given $f=\sum_{m_i \in \NN}  \Delta^{m_i} \in \FF_2[\Delta]$, we group the monomial components of $f$ according to the 2-adic valuation of their exponents. More precisely, let $\nu_2$ denote the 2-adic valuation function, and set $v=\max_i\{\nu_2(m_i)\}$. We write 
\begin{align}
f=\sum_{i=0}^v f_j, \quad \text{with} \quad f_j\coloneqq\sum_{\nu_2(m_i)=j} \Delta^{m_i}.
\end{align}
As an example, $f = \Delta^{16}+ \Delta^{14}+ \Delta^{12}+ \Delta^7+ \Delta^4+\Delta$ can be decomposed into
\[
f = f_4+f_2+f_1+f_0=\Delta^{16}+(\Delta^{12} + \Delta^4) +\Delta^{14}+(\Delta^7 + \Delta).
\]
The following theorem reduces the computation of $d(f)$ to sums of odd powers of $\Delta$.

\begin{thm}\label{general}
Assume the notation above. The following are true:
\noindent
\begin{enumerate}[leftmargin=*]
\item For each $0\le j\le v$, we have 
\[
d(f_j)= j+d\left(\sum_{\nu_2 (m_i) = j} \Delta^{m_i/2^j}\right).
\]
\item The degree of nilpotency of $f$ is given by
\[
d(f)=\max_{0\le j\le v}\{d (f_j)\}.
\]
\end{enumerate}
\end{thm}

\noindent
Important work of Nicolas and Serre \cite{Serre1} computes the degrees of nilpotency of sums of odd powers of $\Delta$ (see Section 2). Together with Theorem \ref{general}, this allows us to calculate the degree of nilpotency of a general cusp form (see Algorithm \ref{alg} and the subsequent example).

At first glance, it seems plausible that there could be infinitely many modular forms mod 2 with a given degree of nilpotency. The next theorem, however, shows that this is not the case by calculating the exact number of such forms.

\begin{thm}\label{count}
If $n$ is a positive integer, then we have
\[
\#\{f\in\FF_2[\Delta]: d(f)= n\}=2^\frac{n(n+1)(n+2)}{6}-2^\frac{(n-1)n(n+1)}{6}.
\]
\end{thm}

So far, all our results concern the entire graded ring $\FF_2[\Delta]$. Presently, determining the limiting distribution of quantities that arise in arithmetic geometry and number theory is an active area of research in the field of arithmetic statistics. In this spirit, it is natural to study Hecke nilpotency on individual spaces $M_k(\ZZ)$ mod 2 and ask how it behaves as $k\to \infty$.To answer this question, we need to consider the subset of cusp forms of some fixed weight $k$ with a given degree of nilpotency. In view of Theorem \ref{count}, one may wish to understand whether the degrees of nilpotency of weight $k$ modular forms tend to some sort of limiting distribution as $k\to \infty$. The next theorem shows that no such limiting behavior exists. 

To ease notation, let $\widetilde{S}\subset \FF_2[\Delta]$ denote the space of cusp forms mod 2, and let $\widetilde{S}_k\subset\FF_2[\Delta]$ denote the space of cusp forms obtained by reducing weight $k$ cusp forms mod 2.

\begin{thm}\label{distribution}  
For a positive even integer $k$, let $m_k\coloneqq\max_{f \in \widetilde{S}_k}\{d(f)\}$ denote the maximal degree of nilpotency realized by cusp forms of weight $k$. Consider the sequence  $\{P_k\}_{k\in\NN}$, where
\[
P_k\coloneqq\frac{\#\{f\in \widetilde{S}_k\mid d(f)=m_k\}}{\card{\widetilde{S}_k}}.
\]
The sequence $\{P_k\}$ does not converge to a limit as $k\to \infty$.
\end{thm}

\begin{rmk}
To prove this theorem, we explicitly construct two sequences of weights $\{k_n\}, \{k_n'\}$ for which $\{ P_{k_n} \} = 1/2$ and $\{ P_{k_n'} \} = 3/4$.
\end{rmk}

The theory of Hecke operators also sheds light on the classical problem of the parity of the partition function $p(n)$, which counts the number of nonincreasing sequences of positive integers which sum to $n$. Perhaps the most famous open problem in the field is the conjecture that half of the partition numbers are odd and the other half are even \cite{Parkin}. Very little is known about this problem; in fact, it is not even known that the number of odd partition values $p(n)$ for $n\le x$ grows at least as fast as $\sqrt{x}$ (see \cite{recordbound}).

It was conjectured by Subbarao \cite{Subbarao} and later proved by Ono \cite{Ono} and Radu \cite{Radu} that, given any arithmetic progression $r\pmod t$, there exist infinitely many integers $M\equiv r \pmod{t}$ such that $p(M)$ is even and infinitely many integers $N \equiv r \pmod{t}$ such that $p(N)$ is odd. Despite such results, it is still a well-known open problem to construct explicit sequences for which the parity of the partition values is known. 

As a first step in this direction, Ono \cite{Boylan} asked for examples of infinitely many nontrivial explicit finite sets of partition values which must contain odd values using results on Hecke nilpotency. As a prototype of such families whose construction does not involve the theory of Hecke nilpotency, recall the famous lemma of Gauss
\[
\sum_{n=0}^\infty q^{\mathcal{T}_n} =\prod_{n=1}^\infty \frac{(1-q^{2n})^2}{1-q^n},
\]
where $\mathcal{T}_n$ denotes the $n$-th triangular number, and Euler's pentagonal number theorem,
\begin{align}\label{pentagonal}
\prod_{n=1}^\infty(1-q^n)=\sum_{n = -\infty}^\infty (-1)^n q^{\frac{n(3n-1)}{2}} = 1+\sum_{n=1}^\infty (-1)^n \left( q^{\frac{n(3n-1)}{2}}+q^{\frac{n(3n+1)}{2}}\right).
\end{align}
The partition numbers $p(n)$ are given by the generating function
\[
F(q)=\prod_{n=1}^\infty \frac{1}{1-q^n}=\sum_{n=0}^\infty p(n) q^n.
\]
Thus, we have
\begin{align*}
\sum_{n=0}^\infty q^{\mathcal{T}_n} \equiv \prod_{n=1}^\infty (1-q^{4n})\cdot \sum_{n=1}^\infty p(n) q^n \pmod 2.
\end{align*}
Equating coefficients, we see from \eqref{pentagonal} that, for every $n$,
\[
p\left(\mathcal{T}_n-\left(6j^2\pm 2j\right)\right)
\]
is odd for at least one $j$. Note that this set is finite because we must have $\mathcal{T}_n- (6j^2 \pm 2j) \geq 0$.

In view of these examples, which can be seen from either a $q$-series or a modular form perspective, we answer Ono's question by finding infinitely many such sets using the theory of Hecke nilpotency. In \cite{Boylan}, Boylan and Ono showed that it is possible to construct squarefree integers $n_s$ as products of \textit{distinct} primes associated with $d(\Delta^{(4^s - 1)/3}) - 1$ Hecke operators which do not annihilate $\Delta^{(4^s - 1)/3}$, such that the set
\[
\left\{ p \left( \frac{n_s \ell^2 - \frac{4^s - 1}{3}}{8} - 4^s \cdot j \right) \bigg \vert \ j \in \ZZ_{\ge 0} \right\}
\]
contains at least one odd element. Their result is purely theoretical and does not give any $n_s$ concretely. Thus, one way to answer Ono's question is to identify explicit integers which play the role of $n_s$. The following theorem achieves this by replacing $n_s$ with special powers of 5, which also shows that the $n_s$ need not be squarefree. In addition, we strengthen the statement by limiting the values of $j$ involved in the preceding set through a natural application of Euler's pentagonal number theorem.

\begin{thm}\label{ns}
Let $S \coloneqq \left\{\frac{n(3n-1)}{2}: n \in \ZZ \right\}$ denote the set of generalized pentagonal numbers. For any $s\in \ZZ_{>1}$ and positive odd integer $\ell$ with $5 \nmid \ell$, there exists some $i$ in the set
\[
\{i\in 2 \ZZ_{\geq 0} + 1: i \leq 2^{s-1} - 1, \ 2^r \neq 2^{s-1} - i - 1 \text{ for all } r = 0, 1, 2, \dots, s-2 \}
\]
such that the finite set
\begin{align}\label{oddset}
\left\{ p \left( \frac{5^{i} \ell^2 - \frac{4^s - 1}{3}}{8} - 4^s \cdot j \right) : \ j \in S \right\}
\end{align}
contains at least one odd element. 
\end{thm}

In Section 5, we shall first present a weaker version of Theorem \ref{ns} (see Proposition \ref{easyset}), where the restriction on $2^{s-1} - i - 1$ is not included. This restriction relies on a combinatorial identity (see Proposition \ref{coeff}), which describes the image of Fourier coefficients of modular forms mod $2$ under the action of iterated Hecke operators. The strength of Theorem \ref{ns}, where this additional restriction is included and rules out certain values of $i$, is illustrated by the second example below.

\begin{exmps}\label{examplens}\quad 

\begin{enumerate}[leftmargin=*]
\item Applying Theorem \ref{ns} to $s = 2$ forces $i$ to be $1$, as the condition on $2^{s-1} - i - 1$ is vacuous. Therefore, the set 
\[ \left\{ p \left( \frac{5 \ell^2 - 5}{8} - 16 j \right) : j \in S \right\}  \]
contains at least one odd element for every odd integer $\ell$ with $5 \nmid \ell$.

\item When $s=3$, the condition in Theorem \ref{ns} on $2^{s-1} - i - 1$ rules out $i=1$, thereby forcing $i$ to be 3. Thus, for every odd integer $\ell$ coprime to 5, the set
\[ 
\left\{ p \left( \frac{125 \ell^2 - 21}{8} - 64 j \right) : j \in S \right\} 
\]
contains some odd element. 
\end{enumerate}
The following two examples further illustrate the second example above for particular values of $\ell$. Note that $S = \{ 0, 1, 2, 5, 7, 12, 15, \dots \}$. 

\begin{enumerate}[leftmargin=*]
\setcounter{enumi}{(2)}
    \item When  $s = 3$ and $\ell = 1$, we have the singleton set $\{p(13) = 101\}$, which provides an amusing proof that $p(13)$ is odd.
    
    \item When $s = 3$ and $\ell = 3$, exactly one of the three elements in the set
\[ \left\{ p(10) = 42, \ p(74) = 7089500, \ p(138) = 12292341831\right\}
\]
is odd, so Theorem \ref{ns} is sharp.
\end{enumerate}
\end{exmps}

The paper is organized as follows. In Section 2, we begin by recalling the important work of Nicolas and Serre in \cite{Serre1}, which gives a method for computing the degree of nilpotency for sums of odd powers of $\Delta$ by decomposing the action of all $T_p$ $(p\ne 2)$ into the action of $T_3$ and $T_5$. In Section 3, we prove Theorem \ref{general} and provide an explicit algorithm for computing the degree of nilpotency of arbitrary modular forms mod 2. In Section 4, we give a formula for the total number of modular forms of any given degree of nilpotency. We then analyze the statistical distribution of weight $k$ cusp forms mod 2 as $k\to \infty$ and show that they exhibit no limiting behavior.
In Section 5, we prove Theorem \ref{ns} by relating the parity of $p(n)$ to the parity of $\tau(n)$ using combinatorial properties of the action of iterated Hecke operators.

\section*{Acknowledgments}  

The authors were participants in the 2022 UVA REU in Number Theory. They are grateful for the support of grants from the National Science Foundation
(DMS-2002265, DMS-2055118, DMS-2147273), the National Security Agency (H98230-22-1-0020), and the Templeton World Charity Foundation. The authors thank Professor Ken Ono for suggesting this problem and for his continued guidance. They would also like to thank Alejandro De Las Penas Castano, Badri Pandey, and Wei-Lun Tsai for valuable discussions.

\section{Work of Nicolas and Serre}

We begin this section by reviewing some key results from Nicolas and Serre \cite{Serre1} on the theory of Hecke nilpotency mod 2. 

For any $m\in \NN$ and any odd prime $p$, we have $(T_p \mid \Delta^m)^2 \equiv T_p \mid \Delta^{2m} \pmod{2}$ by equation \eqref{heckeoperator}. 
Thus, to study the Hecke algebra generated by $T_p$ for odd primes $p$ modulo 2, it suffices to consider the modular forms which, when considered as polynomials in $\Delta$, have only odd exponents. For this reason, Nicolas and Serre focused on the action of $T_p$ on the subspace $\mathcal{F} \subset \FF_2[\Delta]$ generated by $\Delta, \Delta^3, \Delta^5, \ldots$ and introduced the following definitions. 

Let $m$ be a positive integer with binary expansion $m=\sum_{i=0}^\infty \beta_i 2^i$, $\beta_i\in \{0, 1\}$. Then the \textit{code} of $m$ is defined as the pair of non-negative integers $[n_3(m), n_5(m)]$, where
\begin{align} 
n_3(k)\coloneqq\sum_{i=0}^\infty\beta_{2i+1} 2^i,\quad 
n_5(k)\coloneqq\sum_{i=0}^\infty\beta_{2i+2} 2^i.
\end{align}
The \emph{height} of $m$ is defined as $h(m)=n_3(m)+n_5(m)$. For two positive integers $m, n$ of the same parity, we say that $n$ \emph{dominates} $m$ and write $n \succ m$ if $h(m) < h(n)$ or if $h(m) = h(n)$ and $n_5(m) < n_5(n)$. This defines a total order on the set of positive odd (resp. even) integers. 

Using this, Nicolas and Serre defined the code and the height of an element of $\mathcal{F}$ as follows: given any nonzero $f\in \mathcal{F}$, $f$ can be written in the form $\Delta^{m_1} + \Delta^{m_2} +\cdots + \Delta^{m_r}$ with $m_1\succ m_2\succ \cdots \succ m_r$. The \emph{code} of $f$ is defined to be the code of the leading term, i.e., $[n_3(f), n_5(f)]=[n_3(m_1), n_5(m_1)]$. Similarly, the \emph{height} of $f$ is defined to be $h(f)=h(m_1)$.

\begin{exmp}
Let $f=\Delta^{17}+\Delta^{15}+\Delta^7$. Then
\[
[n_3(17), n_5(17)]=[2, 0], \ [n_3(15), n_5(15)]=[3, 1], \ [n_3(7), n_5(7)]=[1, 1].
\]
Thus, we have $[n_3(f), n_5(f)]=[n_3(15), n_5(15)]=[3, 1]$, so $h(f)=h(15)=4.$
\end{exmp}

The following theorem summarizes Nicolas and Serre's method (see Theorem 5.1 in \cite{Serre1}) for computing the degree of nilpotency of elements in $\mathcal{F}$. 

\begin{thm}[Nicolas-Serre]\label{NicSer}
Let $f\in \mathcal{F}, f \neq 0$ be written in the form $\Delta^{m_1} + \Delta^{m_2} +\cdots + \Delta^{m_r}$ with $m_1\succ m_2\succ \cdots \succ m_r$. Then we have
\[
T_3^{n_3(m_1)} T_5^{n_5(m_1)}\mid f \equiv \Delta \pmod{2}.
\]
Moreover, $d(f)=h(f)+1=h(m_1)+1$.
\end{thm}

By Theorem \ref{NicSer}, to determine the degree of nilpotency of sums of odd powers of $\Delta$, it suffices to consider the action of $T_3$ and $T_5$. We wish to remind the reader that, since the constant function 1 is an eigenfunction of $T_2$, the degree of nilpotency of non-cusp forms is infinite.

\begin{defn}\label{newdeg}
Let $f \in \FF_2[\Delta]$ be a cusp form. We define the \emph{degree of nilpotency} of $f$ to be the smallest integer $d=d(f)$ such that
\[
T_{p_1} T_{p_2} \dots T_{p_d} \mid f \equiv 0 \pmod{2}
\]
for any collection of $d$ primes $p_1, \ldots, p_d$. Here we allow $p_i=2$. As in the original definition, if $f\equiv 0\pmod 2$, we set $d(f)\coloneqq-\infty.$ For completeness, if $f \in \FF_2[\Delta]$ is not a cusp form, we set $d(f) \coloneqq \infty$. 
\end{defn}

\begin{rmk}
Since $T_2$ annihilates $\mathcal{F}$, Definition \ref{newdeg} is compatible with the previous definition when restricted to $\mathcal{F}$. Thus, we do not make this distinction in the rest of the article and shall always mean Definition \ref{newdeg} when we mention the degree of nilpotency, unless stated otherwise.
\end{rmk}

\section{Proof of Theorem \ref{general}}

The first part of Theorem \ref{general} concerns the degree of nilpotency of cusp forms which, as polynomials in $\Delta$, have exponents with equal 2-adic valuation. The second part tells us how to assemble the nilpotency degree of a general cusp form $f$ from the nilpotency degrees of its components whose exponents have distinct 2-adic valuations. From here on, we assume the notation established in the paragraph before Theorem \ref{general}.

\begin{proof}[Proof of Theorem \ref{general}]
We first prove part (1). 
Let $f=\sum_{m_i \in 2 \NN}  \Delta^{m_i} \in \widetilde{S}$ be a cusp form.
By equation \eqref{heckeoperator2}, applying $T_2$ is equivalent to halving the exponents of the support for the nonzero coefficients of $f$. That is, 
\[
T_2 \,\bigg| \sum_{m_i \in 2 \NN} \Delta^{m_i} = \sum_{m_i \in 2 \NN} \Delta^{m_i/2}\pmod 2.
\]
More generally, we have
\[
T_2^j \mid f_j\equiv T_2^j\, \bigg| \sum_{\nu_2(m_i)=j}  \Delta^{m_i} \equiv \sum_{\nu_2(m_i) = j}  \Delta^{m_i/2^j} \pmod 2,
\]
which implies that 
\[
d(f_j) \ge d\left(\sum_{\nu_2(m_i) = j}  \Delta^{m_i/2^j}\right)+j.
\]

We show the converse inequality by induction on $j$. When $j = 0$, the statement is clear. Now consider $f_j = \sum_{\nu_2(m_i) = j} \Delta^{m_i}$ and suppose that the claim holds for all polynomials in $\Delta$ of the form $\sum_{\nu(m_i) = k} \Delta^{m_i}$ for $k < j$. If we set $g_j = \sum_{\nu_2(m_i) = j} \Delta^{m_i/2}$, then $f_j(q) \equiv g_j(q^2) \pmod{2}$, since $\Delta^{m_i} (q) \equiv \left(\Delta^{m_i/2}(q)\right)^2 \equiv \Delta^{m_i/2}(q^2)\pmod{2}$.
For $p \ne 2$, the Hecke operator $T_p$ acts on $f$ by
\[
T_p \mid f_j(q) \equiv T_p \mid g_j(q^2) \pmod{2}.
\]
Simple induction shows that
\[
T_{p_1} T_{p_2} \dots T_{p_r} \mid f_j(q) \equiv T_{p_1} T_{p_2} \dots T_{p_r} \mid g_j(q^2) \pmod 2.
\]
Thus, if $p_i\ne 2$ for all $1\le i\le r$, $f_j$ is annihilated exactly when $g_j(q^2)$ is. It follows that $d(f_j) \leq d(g_j) + 1$. By the induction hypothesis, we have $d(f_j) \leq d\left(\sum_{m_i \in \NN}  \Delta^{m_i/2^j}\right)+(j-1)+1=d\left(\sum_{m_i \in \NN}  \Delta^{m_i/2^j}\right)+j$, proving part (1) of the theorem.

For part (2), let  $f, g\in \widetilde{S}$ be any cusp forms as polynomials in $\Delta$, and set $d = \max \{d(f), d(g)\}$. First, we claim that $d(f+g)\le d$ in general. Indeed, by the linearity of the Hecke operators,
\[
T_{p_1} T_{p_2} \cdots T_{p_{d}} \mid (f + g) = (T_{p_1} T_{p_2} \cdots T_{p_{d}} \mid f) + (T_{p_1} T_{p_2} \cdots T_{p_{d}} \mid g) \equiv 0 + 0 \equiv 0 \pmod{2} 
\]
for any collection of primes $p_1, p_2, \dots, p_d$. Iterating this argument, we see that $d(g_1+\ldots+g_r)\le \max\{d(g_1), \ldots, d(g_r)\}$ for any arbitrary finite set $g_1, \ldots, g_r\in \widetilde{S}$. In particular, we see that $d(f_1 + \dots + f_r) \le \max \{ d(f_1), \ldots, d(f_r) \}$

To prove the converse inequality, it suffices to show that, for any pair $f_i$, $f_j$ with $i\ne j$, we have $d(f_i+f_j)\ge d \coloneqq \max\{d(f_i), d(f_j)\}$. To do so, we exhibit a chain of Hecke operators $T_{p_1}, \ldots, T_{p_{d-1}}$ such that $T_{p_1} \ldots T_{p_{d-1}}\mid (f_i+f_j)\not\equiv 0 \pmod{2}$. Without loss of generality, take $i>j$. Since $T_2\mid \Delta^{2n} \equiv \Delta^n \pmod{2}$ for all $n>0$, we have 
\[
T_2^i\mid (f_i+f_j)=T_2^i\mid f_i + T_2^i\mid f_j \equiv T_2^i\mid f_i + 0 \equiv T_2^i \mid f_i \pmod{2} \in \mathcal{F}.
\]
Let $\widehat{f_i} \coloneqq T_2^i\mid f_i$, $\widehat{f_j} \coloneqq T_2^j\mid f_j$. By Theorem \ref{NicSer}, we have $T_3^{n_3(\widehat{f_i})}T_5^{n_5(\widehat{f_i})}T_2^i\mid (f_i + f_j) \equiv \Delta \pmod{2}$, so $d(f_i+f_j)\ge n_3(\widehat{f_i})+n_5(\widehat{f_i})+i+1=d(\widehat{f_i}) + i = d(f_i)$. 

If $n_3(\widehat{f_i})\ge n_3(\widehat{f_j})$ and $n_5(\widehat{f_i})\ge n_5(\widehat{f_j})$, then $d(f_i)>d(f_j)$, and we are done. Otherwise, suppose that $n_3(\widehat{f_i})< n_3(\widehat{f_j})$ or $n_5(\widehat{f_i}) < n_5(\widehat{f_j})$. Then either $T_3^{n_3(\widehat{f_j})}$ or $T_5^{n_5(\widehat{f_j})}$ annihilates $f_i$ modulo $2$, and so
\[
T_3^{n_3(\widehat{f_j})}T_5^{n_5(\widehat{f_j})}T_2^j \mid (f_i+f_j) \equiv  T_3^{n_3(\widehat{f_j})}T_5^{n_5(\widehat{f_j})}T_2^j\mid f_j \equiv \Delta \pmod{2}.
\]
This gives $d(f_i + f_j)\ge n_3(\widehat{f_j})+n_5(\widehat{f_j})+j+1=d(f_j)$. It follows that $d(f) \geq \max\{d(f_i), d(f_j)\}$. By induction, we have $d(f_1 + \dots + f_r) \ge \max \{ d(f_1, \ldots, f_r) \}$, which proves the theorem. 
\end{proof}

We conclude the section with an algorithm which rapidly computes the degree of nilpotency of any cusp form mod 2. 

\begin{alg}[Algorithm for computing $d(f)$\footnote{SAGE code for this algorithm is available at:
https://alejandrodlpc.github.io/files/hecke-nilpotency.ipynb.}]\label{alg}
Let $f=\sum_{m_i \in \NN}  \Delta^{m_i} \in \widetilde{S}$, and suppose that $f\not\equiv 0\pmod 2.$ The degree of nilpotency of $f$ can be computed as follows.

\begin{enumerate}
    \item [\textbf{Step 1.}]  Find the largest non-negative integer $v$ such that $2^v\mid m_i$ for some $m_i\ne 0$, and set $f_j=\sum_{\nu_2(m_i)=j} \Delta^{m_i}$ for each $j = 0, 1, \dots, v$. 
    
    \item[\textbf{Step 2.}] For each $0\le j\le v$, compute
    \[
    d\Big(\sum_{\nu_2 (m_i) = j} \Delta^{m_i/2^j}\Big)=\max_{\nu_2 (m_i) = j} \{h(m_i/2^j)\}+1.
    \]
    
    \item[\textbf{Step 3.}] For each $0 \le j \le v$, compute
    \[ d(f_j)=j+d\Big(\sum_{\nu_2 (m_i) = j} \Delta^{m_i/2^j}\Big).
    \]
    \item[\textbf{Step 4.}]
    Find the degree of nilpotency of $f$ using the formula
\[
d(f)=\max_{0\le j\le v}\{d (f_j)\}.
\]
\end{enumerate}
\end{alg}

\begin{exmp}\label{exmp}
As an example, we compute the degree of nilpotency of 
$
f=\Delta^{128}+\Delta^{60}+\Delta^7+\Delta^4.
$
\begin{enumerate}[]
\item [\textbf{Step 1.}] Rewrite $f$ as $f=f_7+f_2+f_0=\Delta^{128}+(\Delta^{60}+\Delta^4)+\Delta^7=\Delta^{2^7}+(\Delta^{15\cdot 2^2}+\Delta^{2^2})+\Delta^7$, with $v = 7$.
\item[\textbf{Step 2.}] By Theorem \ref{NicSer}, we have
\begin{align*}
&d\bigg(\sum_{\nu_2(m_i) = 7} \Delta^{m_i/2^7}\bigg) = d(\Delta) = \max \{ h(1) \} + 1 = \max \{ 0 \} + 1 = 1, \\ &d\bigg(\sum_{\nu_2(m_i) = 2} \Delta^{m_i/2^2}\bigg) = d(\Delta^{15}+\Delta) = \max \{h(15), h(1) \} + 1 = \max \{ 4, 0 \} + 1 = 5, \\
&d \bigg(\sum_{\nu_2(m_i) = 0} \Delta^{m_i/2^0}\bigg) = d(\Delta^{7}) = h(7) + 1 = 2 + 1 = 3.
\end{align*}
\item[\textbf{Step 3.}] By part (1) of Theorem \ref{general}, we have
\begin{equation*}
d(f_7) = 7 + d (\Delta) = 8, \ 
d(f_2) = 2 + d(\Delta^{15}+\Delta) = 7, \ 
d(f_0) = d(\Delta^7) = 3. 
\end{equation*}
\item[\textbf{Step 4.}] By part (2) of Theorem \ref{general}, we have $d(f)=\max\{d(f_7), d(f_2), d(f_0)\}=\max\{8, 7, 3\}=8.$
\end{enumerate}
\end{exmp}

\section{Distribution of degrees of nilpotency}

We begin by proving a formula for the number of modular forms mod 2 with prescribed degrees of nilpotency under Definition \ref{newdeg}. 
First, we consider the subspace $\mathcal{F}\subset\FF_2[\Delta]$ generated by $\Delta, \Delta^3, \Delta^5, \ldots$.

\begin{prop}
If $n$ is a positive integer, then we have
\[
\#\{f\in \mathcal{F}\mid d(f)= n\}=2^\frac{n(n+1)}{2}-2^\frac{n(n-1)}{2}.
\]
\end{prop}

\begin{proof}
Let $f = \sum_{m_i \in \NN} \Delta^{m_i} \in \mathcal{F}$. We begin by counting the number of exponents $m_i$ such that $n_3(m_i) + n_5(m_i) + 1=n$. Since $0\le n_3(m_i), n_5(m_i)\le n-1$, there are $n$ distinct pairs $[n_3(m), n_5(m)]$ satisfying the preceding equation. By the bijection between the odd natural numbers and $\NN^2$ (via the map $m \mapsto [n_3(m), n_5(m)]$; see Section 4.1 in \cite{Serre1}), there are exactly $n$ modular forms of the form $\Delta^{m_i}$ with degree of nilpotency $n$ and thus $\frac{n(n+1)}{2}$ modular forms of the form $\Delta^{m_i}$ with degree of nilpotency $\le n$. Hence, there are $2^{\frac{n(n+1)}{2}}$ modular forms $f \in \mathcal{F}$ with $d(f) \le n$ and $2^{\frac{n(n-1)}{2}}$ modular forms $f\in \mathcal{F}$ with $d(f) \le n-1$. Subtracting these two quantities proves the formula. 
\end{proof}

This result can be readily extended to the subspace $\widetilde{S} \subset \FF_2[\Delta]$ of cusp forms mod 2, which is the content of  Theorem \ref{count}. 

\begin{proof}[Proof of Theorem \ref{count}]
Let $f = \sum_{m_i \in \NN} \Delta^{m_i} \in \widetilde{S}$. By Theorem 2.3, $d(f)$ is determined by the maximum of $n_3(s_i) + n_5(s_i) + v +1$, where $s_i=m_i / 2^{\nu_2(m_i)}$ and $v=\nu_2(m_i)$.

As above, we begin by counting the number of exponents $m_i$ for which $n_3(s_i) + n_5(s_i) + \nu_2(m_i) + 1=n$.  This is equivalent to finding integral solutions to the inequality $n_3(s_i) + n_5(s_i)\le n-1$; there are $1 + 2 + \dots + (n-1) + n = \frac{n(n+1)}{2}$ such solutions. Since the map \[
m \mapsto [n_3(m/2^{\nu_2(m)}), n_5(m/2^{\nu_2(m)}), \nu_2(m)]
\]
defines a bijection between the natural numbers and $\NN^3$, there are exactly $\frac{n(n+1)}{2}$ modular forms of the form $\Delta^{m_i}$ with degree of nilpotency $n$ and $1 + 3 + 6 + 10 + \dots + \frac{n(n+1)}{2} = \frac{n(n+1)(n+2)}{6}$ modular forms of the form $\Delta^{m_i}$ with degree of nilpotency $ \leq n$. Hence, there are $2^{\frac{n(n+1)(n+2)}{6}}$ modular forms $f \in \FF_2[\Delta]$ with $d(f)\le n$ and thus $2^{\frac{(n-1)n(n+1)}{6}}$ modular forms $f \in \FF_2[\Delta]$ with $d(f)\le n-1$. Subtracting these two quantities establishes the claim.
\end{proof}

Although Theorem \ref{count} gives us an exact formula for the number of cusp forms of given degree of nilpotency, it places no restrictions on the weight $k$. A natural question, then, is to ask for the statistical distribution of weight $k$ modular forms with a given degree of nilpotency as $k\to \infty$.

\begin{lem}\label{dim}
Let $k\ge 4$ be a positive even integer, and let $\widetilde{S}_k\subset\FF_2[\Delta]$ denote the space of cusp forms of weight $k$ mod $2$. Then $\lvert \widetilde{S}_k \rvert =2^{\lfloor k/12\rfloor-1}$ if $k\equiv 2\pmod{12}$ and $\lvert \widetilde{S}_k \rvert =2^{\lfloor k/12\rfloor}$ if $k\not\equiv 2 \pmod{12}$.
\end{lem}

\begin{proof}
Let $f=a_m\Delta^m+\ldots+ a_1\Delta\in \widetilde{S}_k, $ $a_i\in \{0, 1\}$. If $f$ has weight $\le k$, then $m\le \lfloor k/12\rfloor$. For any $1\le i\le m$, the term $a_i\Delta^i$ arises from some modular form of the form 
$a_i E_4^{\alpha_i} E_6^{\beta_i} \Delta^i$
where $a_i$ is any integer and $4\alpha_i+ 6\beta_i+12i=k$. If $k\not\equiv 2\pmod{12}$, then this equation has integral solutions $(\alpha_i, \beta_i, i)$ for all $1\le i\le \lfloor k/12\rfloor$. Since $E_4\equiv E_6\equiv 1\pmod2$, these precisely correspond to the set $\Delta, \Delta^2, \ldots, \Delta^{\lfloor k/12\rfloor}$, which thus forms a basis of $\widetilde{S}_k$. If $k\equiv 2\pmod{12}$, then the preceding equation has integral solutions for all $i<\lfloor k/12\rfloor$, which when reduced mod 2 become $\Delta, \Delta^2, \ldots, \Delta^{\lfloor k/12\rfloor-1}$. Thus, $\widetilde{S}_k$ has as basis $\Delta, \Delta^2, \ldots, \Delta^{\lfloor k/12\rfloor-1}$ if $k\equiv 2\pmod {12}$ and $\Delta, \Delta^2, \ldots, \Delta^{\lfloor k/12\rfloor}$ if $k\not\equiv 2\pmod {12}$.
It follows that $\widetilde{S}_k$ has dimension $\lfloor k/12\rfloor-1$ if $k\equiv 2\pmod{12}$ and dimension $\lfloor k/12\rfloor$ otherwise, so $\lvert \widetilde{S}_k \rvert =2^{\lfloor k/12\rfloor-1}$ if $k\equiv 2\pmod{12}$ and $\lvert \widetilde{S}_k \rvert =2^{\lfloor k/12\rfloor}$ if $k\not\equiv 2 \pmod{12}$, as desired.
\end{proof}

\begin{rmk}
Theorem \ref{count}, which at first glance only applies to the entire graded ring $\FF_2[\Delta]$, can be applied more broadly as a consequence of Lemma \ref{dim}. For a fixed weight $k$, the formula in Theorem \ref{count} gives the total number of modular forms $f \in \FF_2[\Delta]$ with $d(f) = n$ when $n$ is sufficiently small, in the sense that all monomials $\Delta^m$ with $d(\Delta^m) = n$ are in the basis of $\widetilde{S}_k$. For positive integers $n$ which do not meet this condition, the total number of modular forms $f \in \FF_2[\Delta]$ with $d(f) = n$ is a sum of powers of $2$, arising from the monomials $\Delta^m$ with $d(\Delta^m) = n$ that \textit{are} in the basis of $\widetilde{S}_k$. This number is strictly smaller than $2^{\frac{n(n+1)(n+2)}{6}} - 2^{\frac{(n-1)n(n+1)}{6}}$.
\end{rmk}

\begin{proof}[Proof of Theorem \ref{distribution}] Recall that, for a positive even integer $k\ge 4$, we let $m_k=\max_{f \in \widetilde{S}_k}\{d(f)\}$ denote the maximal degree of nilpotency realized by cusp forms of weight $k$. We construct two increasing subsequences of weights $\{k_n\}_{n\in\NN}$ and  $\{k_n'\}_{n\in\NN}$ such that
\[
\frac{\#\{f\in \widetilde{S}_{k_n}\mid d(f)=m_{k_n}\}}{\#\widetilde{S}_{k_n}}=\frac{1}{2}
\]
for each $k_n$ and 
\[
\frac{\#\{f\in \widetilde{S}_{k_n'}\mid d(f)=m_{k_n'}\}}{\card{\widetilde{S}_{k_n'}}}=\frac{3}{4}
\]
for each $k_n'$.

For the first sequence, we consider a subsequence of weights at which a new degree of nilpotency appears; we take $\{k_n\}$ where 
\begin{align}\label{firtapp}
d(\Delta^{\lfloor k_n/12\rfloor})=n,\ d(\Delta^r)< n \text{ for all } r< \lfloor k_n/12\rfloor,
\end{align}
and $k_n\not\equiv 2\pmod{12}$. 
For simplicity, we only consider weights $k_n$ such that $12\mid k_n$. For example, $d(f)=2, 3, 4, 5, 6$ are first attained by $f=\Delta^2, \Delta^4, \Delta^8, \Delta^{15}, \Delta^{27}$, which correspond to weights $k=24, 48, 96, 180, 324$.

Let $f\in \widetilde{S}_{k_n}$, and suppose that $d(f)=m_{k_n}=n$. Since $d(\Delta^r)<n$ for any $r<k_n/12$, the expansion of $f$ must contain $\Delta^{k_n/12}$ by Theorem \ref{general}. Thus, Lemma \ref{dim} implies that
\[
\frac{\#\{f\in \widetilde{S}_{k_n}\mid d(f)=m_{k_n}\}}{\card{\widetilde{S}_{k_n}}}=\frac{2^{k_n/12-1}}{2^{k_n/12}}=\frac{1}{2}.
\]

We now give an explicit formula for a subsequence of $\{k_n\}$ with the above properties. The values of $k_n$ can be determined by Algorithm \ref{alg}. Consider the sequence $\{k_n\} =\{12, 96, 516, 2052, 8196,\ldots\}$, where
\[
\frac{k_1}{12} \coloneqq 1, \ \frac{k_2}{12} \coloneqq 8, \ \frac{k_n}{12} \coloneqq 1+\sum_{i=0}^{n-1}2^{2i+1}  \text{ for }n\ge 3.
\]
Clearly, $k_1$ and $k_2$ satisfy (\ref{firtapp}). To lighten the notation, let $a_n=k_n/12$. By Theorem 2.1, if $n \geq 3$, we have $d(\Delta^{a_n})=2^n$. It remains to be shown that $d(\Delta^m)<d (\Delta^{a_n})$ for all $m<a_n$. 

We proceed by induction on $n$. The base case $a_3=43$ can be easily checked. For the inductive hypothesis, suppose that for all $i\le n-1$, we have $d(\Delta^m)<d (\Delta^{a_i})$ for any $m<a_i$. Note that $2^{2i-1}< a_i< 2^{2i}$. The following lemma lets us compute the degree of nilpotency of certain powers of $\Delta$ from smaller powers.

\begin{lem}\label{deltabinary}
Let $n>1$. If $m<2^{2n+1}$, then $d(\Delta^{m+2^{2n+1}})=d(\Delta^m)+2^n$.
\end{lem}
\begin{proof}
Comparing the binary expansions of $m+2^{2n+1}$ and $m$, we see that $n_3(m+2^{2n+1})=n_3(m)+2^n$ and $n_5(m+2^{2n+1})=n_5(m)$.
\end{proof}
 
First, suppose that $m$ is odd. Since $a_n=a_{n-1}+2^{2n-1}$, it follows that $m-2^{2n-1}<a_{n-1}$. Combining Lemma \ref{deltabinary} and the inductive hypothesis, we have
\[
d(\Delta^m)-2^n=d(\Delta^{m-2^{2n-1}}) <d(\Delta^{a_{n-1}})=d(\Delta^{a_n})-2^n,
\]
completing the inductive argument in this case.

Next, suppose that $m<a_n$ is even, and let $v=\nu_2(m)$. Note that $v<n$. We claim that $d(\Delta^{m/2^v})< d(\Delta^{2^{2n-v}-1})$. To see this, note that since $a_n < 2^{2n}$, we have $m/2^v<2^{2n-v}$, so the binary expansion of $m/2^v$ has strictly fewer digits than that of $2^{2n-v}$. Comparing the binary expansions of $m/2^v$ and $2^{2n-v}-1$ yields $h(m/2^v) < h (2^{2n-v}-1)$, which implies that $d(\Delta^{m/2^v})<d(\Delta^{2^{2n-v}-1})$.
If $2n-v$ is odd, it may be rewritten as $2\left(n-\frac{v-1}{2}\right)-1$, and so $d(\Delta^{2^{2n-v}-1})=2^{n-(v-1)/2} - 1$. It follows that $d(\Delta^m)< 2^{n-(v-1)/2}+v-1$. For all $n\ge 3$ and $0 < v<n$ with $v$ odd, we have  $2^{n-(v-1)/2}+v-1 \leq 2^n$, which shows that $d(\Delta^m)< d(\Delta^{a_n})$. The case where $2n-v$ is even is proved analogously using the fact that $d(\Delta^{2^{2n-v}-1})=2^{n-v/2}+2^{n-v/2-1}-1 \leq 2^n$. This completes the induction.
Thus, $k_n$ satisfies (\ref{firtapp}) for all $n\ge 1$, which justifies the choice of $\{k_n\}$.

We now construct the subsequence $\{k_n'\}$ for which $P_{k_n'}$ is the constant sequence equaling $3/4$. Consider the subsequence $\{k_n'\}=\{72, 348, 540, 2076, \ldots\}$,
where
\[
\frac{k_1'}{12} \coloneqq 6, \ \frac{k_2'}{12} \coloneqq 29, \ \frac{k_n'}{12} \coloneqq \frac{k_n}{12}+2 \text{ for } n\ge 3,
\]
i.e., $k_n'=k_n+24$ for $n\ge 3$. It is straightforward to verify that $P_{72}=P_{348}=3/4$. 

Now suppose that $n\ge 3$. By construction, $k_n$ is the lowest weight at which $d(f)=2^n$ is attained. Since $k_n'/12=k_n/12+2$ is odd and has code $[2^n-2, 1]$, we have $d(\Delta^{k_n/12+2})=d(\Delta^{k_n'/12})=2^n$, which is the maximum possible degree of nilpotency for weights $k_n$ and $k_n'$. Moreover, $1+k_n/12$ is even with $\nu_2(1+k_n/12)=2$, and $(1+k_n/12)/4$ has code $[2^{n-1} -1, 0]$. Thus, for fixed $k_n'$, the set $\{f\in \widetilde{S}_{k_n'}: d(f)=m_{k_n'}=2^n\}$ consists exactly of the cusp forms 
\begin{align*}
    \Delta^{k_n/12}+\sum_{i=1}^{k_n/12-1}a_i \Delta^i, \quad \Delta^{k_n/12+1}+\Delta^{k_n/12}+\sum_{i=1}^{k_n/12-1}a_i \Delta^i,  \quad
    \Delta^{k_n/12+2}+\sum_{i=1}^{k_n/12+1}b_i \Delta^i,
\end{align*}
where $a_i, b_i\in\FF_2$. It follows that
\[
P_{k_n'}=\frac{\#\{f\in \widetilde{S}_{k_n'}\mid d(f)=m_{k_n'}\}}{\card{\widetilde{S}_{k_n'}}}=\frac{2^{k_n/12-1}+2^{k_n/12-1}+2^{k_n/12+1}}{2^{k_n/12+2}}=\frac{3}{4},
\]
as desired.
\end{proof}

\begin{rmks}
Heuristically, the choice of $\{k_n\}$ reflects the observation that weights with associated codes $[n_3, 0]$ are the ``cheapest'' way to produce cusp forms with degrees of nilpotency $2^n$ and therefore correspond to the initial appearance of cusp forms of these nilpotency degrees.

The sequence $\{k_n\}$ does not include all the weights at which a new degree of nilpotency appears; we selected this particular subsequence because its terms can be conveniently given by a closed-form formula. The subsequence $\{k_n''\}=\{2^{2n}-1\}$ also consists of weights at which a new nilpotency degree appears, but it is easier to construct the constant sequence $\{k_n'\}$ from the sequence $\{k_n\}$ as defined above.

Moreover, one can extract other sequences of weights $\{\Tilde{k}_n\}$ for which $P_{\Tilde{k}_n}$ is a constant sequence. In fact, there are infinitely many distinct sequences $P_{\Tilde{k}_n}$ which are constant. These values themselves form a convergent sequence, whose limit the interested reader may compute.
\end{rmks}

\begin{exmp}
The following histograms, normalized to have mass 1, contrast the distributions of cusp forms of weight $96$ and $72$. The $x$-axis is the degree of nilpotency, and the $y$-axis is the proportion of cusp forms of each degree of nilpotency.

\begin{figure}[H] 
    \centering
    \subfloat[$P_{96}=1/2$]{%
        \includegraphics[width=0.45\textwidth]{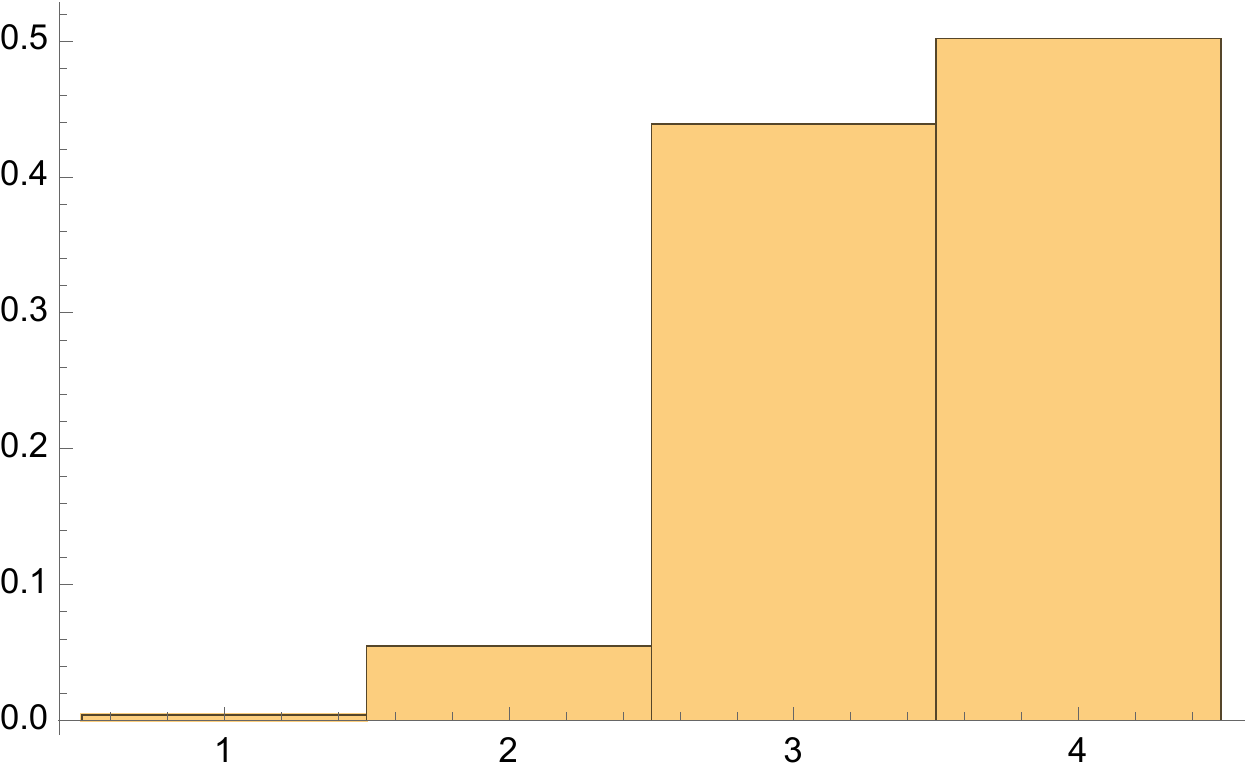}%
        \label{fig:a}%
        }%
    \hfill%
    \subfloat[$P_{72}=3/4$]{%
        \includegraphics[width=0.45\textwidth]{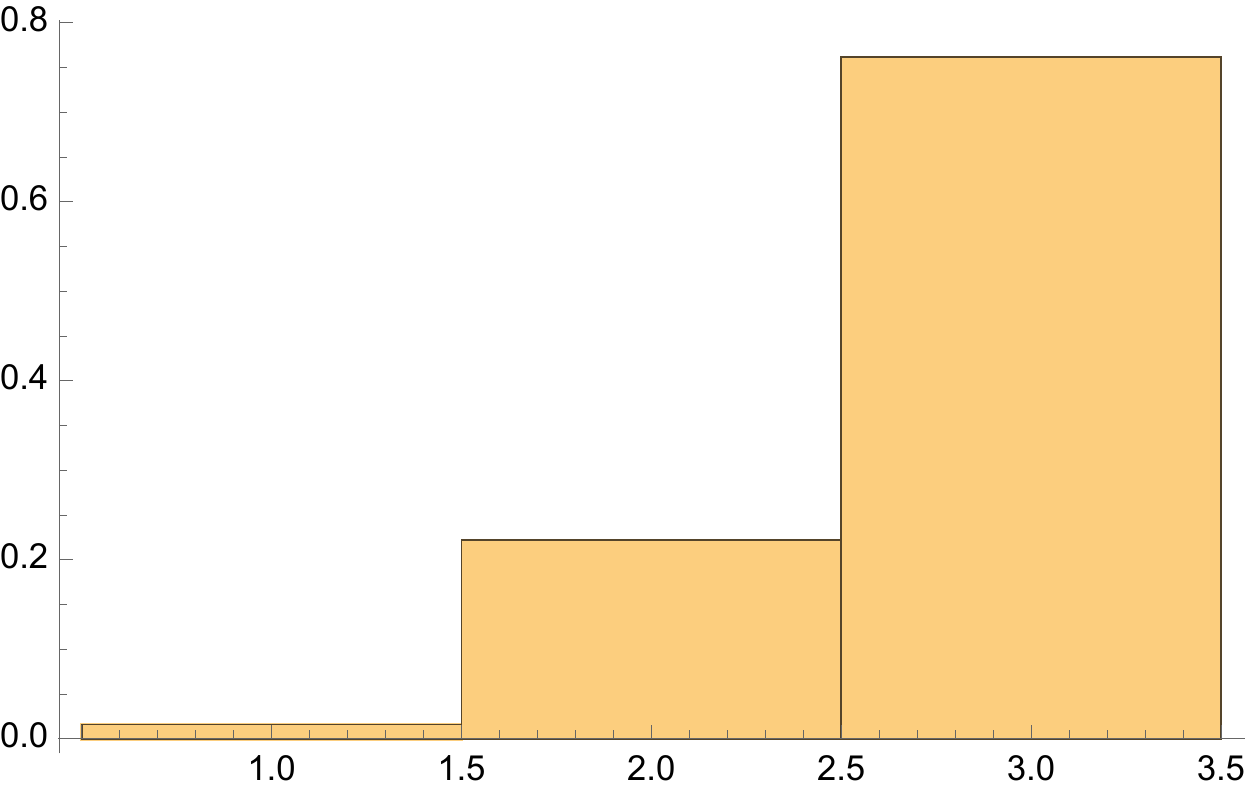}
        \label{fig:b}%
        }%
    \caption{Histograms for weights 96 and 72.}
\end{figure}
\end{exmp}

\section{Application to Ono's Problem}

We begin by fixing some notations. Given a modular form $f \in M_k(\ZZ)$, let $a_f(n)$ denote the $n$-th Fourier coefficient of $f$. We will suppress the subscript $f$ and write $a_f(n)=a(n)$ when $f$ is clear from the context. The following lemma, which is a generalization of Lemma 2.1 in \cite{Boylan}, relates the parity of partition values to the parity of Fourier coefficients of powers of $\Delta$. Recall that $S$ is the set of generalized pentagonal numbers.

\begin{lem} \label{lem1} Let $m$, $n$, and $24r$ be positive integers, and suppose that $24(m+r)$ and $24r$ are powers of 2. Then we have
\[
a_{\Delta^m}(n) \equiv p \left( \frac{n - m}{24r} \right) + \sum_{j \in S} p \left( \frac{n - m}{24r} - 24(m+r) j \right) \pmod{2}.
\]
\end{lem}

\begin{proof}
We have
\begin{align}\label{convolution}
    \Delta^m&= \left( q \prod_{n=1}^\infty (1-q^n)^{24} \right)^m
   \notag \equiv q^m \prod_{n=1}^\infty\left(1-q^{24(m+r)n}\right)\prod_{n=1}^\infty  \frac{1}{(1-q^{24 r n})}\pmod{2}\\
    \notag &\equiv \left( \sum_{h=-\infty}^\infty (-1)^h q^{24(m+r) \cdot \frac{h(3h-1)}{2}} \right) \left( \sum_{k=0}^\infty p(k) q^{24 r k + m} \right) \pmod{2},
\end{align}
where the last equality follows from \eqref{pentagonal}. The claim follows by comparing coefficients.
\end{proof}

Using Lemma \ref{lem1} and Hecke nilpotency, we obtain a weak version of Theorem \ref{ns}.

\begin{prop}\label{easyset}
For every positive integer $s$ and positive odd integer $\ell$ with $5\nmid \ell$, there exists some element $i$ in the set 
\[ 
\{ i \in 2 \ZZ_{\geq 0} + 1 : i \leq 2^{s-1} - 1 \}
\]
such that
\[
\left\{ p \left( \frac{5^{i} \ell^2 - \frac{4^s - 1}{3}}{8} - 4^s \cdot j \right) : \ j \in S \right\}
\]
contains at least one odd element.
\end{prop}

\begin{proof}
We use Lemma \ref{lem1} with $m = (4^s - 1)/3$ and $r = 1/3$. First, we compute the code of $m$. Using the geometric sum $(4^s - 1)/3= 1 + 2^2 + 2^4 + \dots + 2^{2s - 2}$, we have 
\[
n_3(m) = 0, \quad n_5\left(m\right) = 1 + 2^1 + 2^2 + \dots + 2^{s-2} = 2^{s-1} - 1.
\]
Thus, Theorem \ref{NicSer} implies that
\begin{align}\label{heckepower}
    T_5^{2^{s-1} - 1} \mid \Delta^{\frac{4^s-1}{3}} \equiv \Delta \pmod{2}.
\end{align}
By \eqref{heckeoperator},  the $n$-th coefficient of $T_5^{2^{s-1} - 1}\mid \Delta^m$ is a linear combination of the set \[ 
\left\{ a_{\Delta^m}(5^{2^{s-1} - 1} n), a_{\Delta^m}(5^{2^{s-1} - 3}n), \dots, a_{\Delta^m}(5 n), a_{\Delta^m}\left(\frac{n}{5}\right), \dots, a_{\Delta^m} \left(\frac{n}{5^{2^{s-1} - 1}}\right) \right\}. 
\]
For $n=\ell^2$, equations \eqref{oddsq} and \eqref{heckepower} imply that there exist some $c_i \in \FF_2$ for each odd integer $0 < i \leq 2^{s-1}$ such that
\[
c_{2^{s-1}-1} a_{\Delta^m}(5^{2^{s-1}-1} \ell^2) + \dots + c_{-2^{s-1}+1} a_{\Delta^m}(5^{-2^{s-1}+1}\ell^2) \equiv 1 \pmod{2}.
\]
If, in addition, we have $5 \nmid \ell$, then 
\[
c_{2^{s-1}-1} a_{\Delta^m}(5^{2^{s-1}-1} \ell^2) +  \dots + c_1 a_{\Delta^m}(5\ell^2) \equiv 1 \pmod{2}.
\]
Thus, there exists a positive odd integer $i_0 \le 2^{s-1} - 1$ such that $c_{i_0} a_{\Delta^m} (5^{i_0} \ell^2)$ is odd. It follows from Lemma \ref{lem1} that 
\[
p \left( \frac{5^{i_0} \ell^2 - \frac{4^s - 1}{3}}{8} - 4^s j  \right)
\]
must be odd for some $j \in S$. 
\end{proof}

In the proof of Proposition \ref{easyset}, we exploit explicit linear combinations of Fourier coefficients which sum to an odd number. However, it turns out that many of the terms in these linear combinations will be even, and the point of Theorem \ref{ns} is to rule out those terms which do not contribute to the odd parity. This analysis will lead to the additional condition that $i\ne  2^m- 2^{s-1} - 1$ for any $m \ge 0$. To this end, we first establish some notations.

Throughout the rest of this section, we will let $p$ denote an odd prime and $n$ a positive integer. Recall that the operators $U_p$ and $V_p$ (see \cite{UPVP}) are defined by 
\begin{align}\label{operatoruv}
U_p \mid f \coloneqq \sum_{n=0}^\infty a(pn) q^n, \quad V_p\mid  f \coloneqq \sum_{n=0}^\infty a(n) q^{pn}.
\end{align} 
Equation \eqref{heckeoperator} implies that $T_p \equiv U_p + V_p \pmod{2}$, and thus $a_{T_p^t\mid f}(n) \equiv a_{(U_p+V_p)^t\mid f}(n)\pmod 2$ for all $n$. The $n$-th Fourier coefficient of $(U_p + V_p)^t\mid f$ is a $\ZZ$-linear combination of elements in  $\{a_f(p^i n): -t\le i\le t, t\equiv i\pmod 2\}$. Using this fact, we define the sequence $\{ c_{p,t,i} (n) \}$ by
\begin{align}\label{cptidef}
a_{(U_p+V_p)^t \mid f}(n) & = \sum_{-t \leq i \leq t, \ i\equiv t\Mod 2 } c_{p,t,i}(n) a_f(p^i n), \quad c_{p,t, i}(n) \in \ZZ.
\end{align}
If we set $c_{p,t,i}(n) \coloneqq 0$ whenever $p^{-i} \nmid n$, then $\{ c_{p,t,i}(n) \}$ is uniquely determined by the process of applying $(U_p+V_p)$ $t$ times. For completeness, we also set $c_{p,t,i}(n) \coloneqq 0$ if $|i|>t$.
In what follows, we treat $a_f(p^i n)$ as formal symbols. We start by presenting a recursive relation for $c_{p,t,i}(n)$. 

\begin{lem}\label{recurs}
Given $t, n \in \ZZ^+$, we have
\begin{equation}\label{recursion}
c_{p,t,i}(n)=\begin{cases}
 c_{p,t-1,i-1}(n) + c_{p,t-1, i+1}(n)  & \text{ if } {\max \{ -\nu_p(n), -t \} \leq i \leq t} \\
0  & \textrm{ if }  i < \max \{ -\nu_p(n), -t \} \text{ or } i > t .
\end{cases}
\end{equation}
\end{lem}

\begin{proof} 
The second case follows immediately from the condition that $c_{p,t,i}(n) \coloneqq 0$ when $p^{-i} \nmid n$ or $|i| > t$, so we focus on integers $i \leq t$ for which $p^{-i} \mid n$.
Notice that 
 \begin{align}\label{heckerelation}
a_{(U_p + V_p)^{t}\mid f }(n)=a_{(U_p + V_p)^{t-1}\mid f}(pn)+a_{(U_p + V_p)^{t-1}\mid f}(n/p)
\end{align} 
for all $t$ in the case where all three coefficients are nonzero, which holds when $p^{-i+1} \mid n$ and $i < t$. Equating coefficients and using equation \eqref{cptidef}, we have
\begin{equation}\label{coeffidentity} 
c_{p,t,i}(n)=c_{p,t-1,i-1}(n) + c_{p,t-1, i+1}(n)
\end{equation}
when $\nu_p(n) < i < t$.

The boundary values $c_{p,t-1, t+1}(n)$ and $c_{p,t-1, -\nu_p(n) - 1}(n)$ are set to zero by definition.  Additionally, the definition of $c_{p,t, i}(n)$ and equation \eqref{heckerelation} imply that
\begin{align*} 
c_{p,t, t}(n) = c_{p,t-1, t-1}(n),\quad  c_{p,t,-\nu_p(n)}(n) = c_{p,t-1, -\nu_p(n) + 1}(n),
\end{align*}
so the statement also holds for $i = -\nu_p(n)$ and $i = t$.
\end{proof}

Using Lemma \ref{recurs}, we can easily calculate $c_{p,t,i}(n)$ in the special case where $p^t \mid n$.

\begin{prop} For $t, n \in \ZZ^+$ and odd prime $p$, we have $\displaystyle c_{p,t,i}(n) = \binom{t}{\frac{t-i}{2}}$.
\end{prop}

\begin{proof}
For all integers $t' \leq t$, Lemma \ref{recurs} implies that
\begin{equation}\label{coef2}
c_{p,t',i}(n)=c_{p,t'-1,i-1}(n) + c_{p,t'-1, i+1}(n)
\end{equation}
for any integer $i$ such that $i \equiv t' \pmod{2}$ and $-t' \leq i \leq t'$.
Now consider Pascal's identity  $\binom{x}{y}=\binom{x-1}{y-1} + \binom{x-1}{y}$. Setting $x = t', y = \frac{t' - i}{2}$ yields
\begin{align}\label{recur2}
\binomd{t'}{\frac{t'-i}{2}}=\binomd{t'-1}{\frac{t' - i}{2} -1} + \binomd{t'-1}{\frac{t' - i}{2}}. 
\end{align}
Since the base cases $c_{p,1,1}(n)=c_{p,1,-1}(n) = 1$ and $\binom{1}{\frac{1-1}{2}}=\binom{1}{\frac{1-(-1)}{2}} = 1$ are equal, comparing the recursive relations in \eqref{coef2} and \eqref{recur2} shows that two sequences $\{ c_{p,t,i}(n) \}$, $\left\{ \binom{t}{\frac{t-i}{2}} \right\}$ are identical.
\end{proof}

More care is needed, however, when $p^t \nmid n$. As shown in Lemma \ref{recurs}, the equality presented in \eqref{coeffidentity} does not hold if $i= -\nu_p(n) - 1$, as in this case the left-hand side is zero while the right-hand side is nonzero. This particular exception suggests that one often has $c_{p,t,i}(n)<\binom{t}{(t-i)/2}$. To compute $c_{p,t,i}(n)$, we need to account for the overcounted contribution from each term associated with $c_{p,t',-\nu_p(n)-1}$ where $t' \le t$ and $t' \equiv -\nu_p(n) - 1 \pmod{2}$. The coefficients $c_{p,t, i}(n)$ can be modelled by a modified Pascal's triangle, where entries of the form $\binom{t}{(t+\nu_p(n) + 1)/2}$ are manually replaced by zeros. For example, if $\nu_p(n) = 0$, one may think of $c_{p,t,i}(n)$ as the $i$-th entry in the $t$-th row of the following array:

\begin{figure}[!h]
\begin{center}
\resizebox{18em}{!}{\begin{tabular}{*{25}{c}}
&&&&&&&&1&&&&&&&\\[1mm]
&&&&&&&1&&0&&&&&\\[1mm]
&&&&&&1&&1&&0&&&&\\[1mm]
&&&&&1&&2&&0&&0&&&\\[1mm]
&&&&1&&3&&2&&0&&0&&\\[1mm]
&&&1&&4&&5&&0&&0&&0&\\[1mm]
&&1&&5&&9&&5&&0&&0&&0\\[1mm]
\end{tabular}}
\end{center}
\caption{Values of $c_{p,t,i}(n)$ with $\nu_p(n) = 0$.}
\label{table:pascal}
\end{figure} 

Our goal is to obtain a closed formula for $c_{p,t,i}(n)$ which only involves binomial coefficients, whose parity is well known. Here, we illustrate this process with the following an example, where we demonstrate how one obtains Figure 2 in the case $\nu_p(n) = 0$.

\begin{exmp}
Choose $p, n$ such that $\nu_p(n) = 0$. We calculate $c_{p,t,i}(n)$ for $t \leq 6$. To obtain the array in Figure 2, we replace the entries on the ``critical column'' $i = -\nu_p(n) - 1 = -1$ by zeros in an iterative manner, starting with the entry associated with $t = 1$, then $t = 3$, $t = 5$, and so on, noting that we only need to consider odd values of $t$ because $i = -1$ is odd.

When we replace an entry with a zero, we must also remove its contribution to each entry in subsequent rows. This is modeled by subtracting shifted multiples of Pascal's triangle. First, we replace the entry associated with $t = 1$ and $i = -1$: 
\begin{table}[h!]
\resizebox{12em}{!}{
\begin{tabular}{*{25}{c}}
&&&&&&1&&&&&&&\\[1mm]
&&&&&1&&\textbf{1}&&&&&\\[1mm]
&&&&1&&2&&1&&&&\\[1mm]
&&&1&&3&&3&&1&&&\\[1mm]
&&1&&4&&6&&4&&1&&\\[1mm]
&1&&5&&10&&10&&5&&1&\\[1mm]
1&&6&&15&&20&&15&&6&&1\\[1mm]
\end{tabular}}
\textbf{--}
\resizebox{12em}{!}{
\begin{tabular}{*{25}{c}}
&&&&&&0&&&&&&&\\[1mm]
&&&&&0&&\textbf{1}&&&&&\\[1mm]
&&&&0&&1&&1&&&&\\[1mm]
&&&0&&1&&2&&1&&&\\[1mm]
&&0&&1&&3&&3&&1&&\\[1mm]
&0&&1&&4&&6&&4&&1&\\[1mm]
0&&1&&5&&10&&10&&5&&1\\[1mm]
\end{tabular}}
\textbf{=}
\resizebox{12em}{!}{
\begin{tabular}{*{25}{c}}
&&&&&&1&&&&&&&\\[1mm]
&&&&&1&&\textbf{0}&&&&&\\[1mm]
&&&&1&&1&&0&&&&\\[1mm]
&&&1&&2&&1&&0&&&\\[1mm]
&&1&&3&&3&&1&&0&&\\[1mm]
&1&&4&&6&&4&&1&&0&\\[1mm]
1&&5&&10&&10&&5&&1&&0\\[1mm]
\end{tabular}}
\end{table}

\noindent Next, we replace the entry associated with $t = 3$ and $i = -1$: 
\begin{table}[h!]
\resizebox{12em}{!}{
\begin{tabular}{*{25}{c}}
&&&&&&1&&&&&&&\\[1mm]
&&&&&1&&0&&&&&\\[1mm]
&&&&1&&1&&0&&&&\\[1mm]
&&&1&&2&&\textbf{1}&&0&&&\\[1mm]
&&1&&3&&3&&1&&0&&\\[1mm]
&1&&4&&6&&4&&1&&0&\\[1mm]
1&&5&&10&&10&&5&&1&&0\\[1mm]
\end{tabular}}
\textbf{--}
\resizebox{12em}{!}{
\begin{tabular}{*{25}{c}}
&&&&&&0&&&&&&&\\[1mm]
&&&&&0&&0&&&&&\\[1mm]
&&&&0&&0&&0&&&&\\[1mm]
&&&0&&0&&\textbf{1}&&0&&&\\[1mm]
&&0&&0&&1&&1&&0&&\\[1mm]
&0&&0&&1&&2&&1&&0&\\[1mm]
0&&0&&1&&3&&3&&1&&0\\[1mm]
\end{tabular}}
\textbf{=}
\resizebox{12em}{!}{
\begin{tabular}{*{25}{c}}
&&&&&&1&&&&&&&\\[1mm]
&&&&&1&&0&&&&&\\[1mm]
&&&&1&&1&&0&&&&\\[1mm]
&&&1&&2&&\textbf{0}&&0&&&\\[1mm]
&&1&&3&&2&&0&&0&&\\[1mm]
&1&&4&&5&&2&&0&&0&\\[1mm]
1&&5&&9&&7&&2&&0&&0\\[1mm]
\end{tabular}}
\end{table}

\noindent
Finally, we replace the entry associated with $t = 5$ and $i = -1$. Since any additional term which will be replaced by a zero under this scheme corresponds to $t > 6$, the following step concludes our calculation of $c_{p,t,i}(n)$ for $t \leq 6$: 
\begin{table}[h!]
\resizebox{12em}{!}{
\begin{tabular}{*{25}{c}}
&&&&&&1&&&&&&&\\[1mm]
&&&&&1&&0&&&&&\\[1mm]
&&&&1&&1&&0&&&&\\[1mm]
&&&1&&2&&0&&0&&&\\[1mm]
&&1&&3&&2&&0&&0&&\\[1mm]
&1&&4&&5&&\textbf{2}&&0&&0&\\[1mm]
1&&5&&9&&7&&2&&0&&0\\[1mm]
\end{tabular}}
\textbf{--}
\resizebox{12em}{!}{
\begin{tabular}{*{25}{c}}
&&&&&&0&&&&&&&\\[1mm]
&&&&&0&&0&&&&&\\[1mm]
&&&&0&&0&&0&&&&\\[1mm]
&&&0&&0&&0&&0&&&\\[1mm]
&&0&&0&&0&&0&&0&&\\[1mm]
&0&&0&&0&&\textbf{2}&&0&&0&\\[1mm]
0&&0&&0&&2&&2&&0&&0\\[1mm]
\end{tabular}}
\textbf{=}
\resizebox{12em}{!}{
\begin{tabular}{*{25}{c}}
&&&&&&1&&&&&&&\\[1mm]
&&&&&1&&0&&&&&\\[1mm]
&&&&1&&1&&0&&&&\\[1mm]
&&&1&&2&&0&&0&&&\\[1mm]
&&1&&3&&2&&0&&0&&\\[1mm]
&1&&4&&5&&0&&0&&0&\\[1mm]
1&&5&&9&&5&&0&&0&&0\\[1mm]
\end{tabular}}
\end{table}

\noindent Indeed, this is the array displayed in Figure 2. 
\end{exmp}

To formalize this idea, define 
\[
\widehat{c}_{p, t, -\nu_p(n) - 1}(n) \coloneqq c_{p, t-1,-\nu_p(n)}(n) + c_{p, t-1, -\nu_p(n) - 2}(n).
\] 
In the modified Pascal's triangle model, $\widehat{c}_{p, t,-\nu_p(n) - 1}(n)$ corresponds to the entry associated with $i = -\nu_p(n) - 1$ and $t$ if we were to set $c_{p, t', -\nu_p(n) - 1}(n) \coloneqq 0$ for all integers $t'$ strictly less than $t$. Then $\widehat{c}_{p, t,-\nu_p(n) - 1}(n)$ is the value of the entry at $t$ and $i = -\nu_p(n) -1$ one step before it is replaced by a zero. For instance, in the above example, we have $\widehat{c}_{p, 1, -1}(n) = 1$, $\widehat{c}_{p, 3, -1}(n) = 1$, and $\widehat{c}_{p, 5, -1}(n) = 2$. 

For integers $t \equiv -\nu_p(n) - 1 \pmod{2}$, we can express $\{ c_{p, t,i}(n) \}$ in terms of $\{\widehat{c}_{p, t, -\nu_p(n) - 1}(n) \}$. This is made precise by the following lemma.

\begin{lem}\label{recursionchat}
Assuming the notation above, we have
\[ c_{p,t,i}(n) = \binom{t}{\frac{t-i}{2}} - \sum_{\substack{t' \equiv \nu_p(n) + 1\!\Mod{2}, \\ t' \leq t}} \widehat{c}_{p, t', -\nu_p(n) - 1}(n) \binomd{t - t'}{\frac{t-i}{2} - \frac{t'+\nu_p(n)+1}{2}}.
\]
\end{lem}

\begin{proof}
To ease notation, we shall abbreviate $c_{p,t,i}(n)$ (resp. $\widehat{c}_{p,t,i}(n)$) to $c_{t,i}$ (resp. $\widehat{c}_{t,i}$) in what follows. In addition, let $v\coloneqq\nu_p(n)$. 

To find $c_{t,i}$, we need to subtract the overcounted contribution from each entry with $t' \leq t$ and $i = -v - 1$. In view of the Pascal's triangle model, this is given by the product of $\widehat{c}_{t', -v - 1}$ and the value that this entry would have added to $c_{t,i}$ if we have had $\widehat{c}_{t', -v - 1}=1$. The latter quantity can be expressed as a binomial coefficient of differences associated with the entries: if no correction were required, the binomial coefficient associated with $c_{t,i}$ would have been $\binom{t}{(t-i)/2}$, and the binomial coefficient associated with $c_{t', -v - 1}$ would have been $\binom{t}{(t + v + 1)/2}$. Taking the entry at $t = t'$ and $i =-v - 1$ to be the apex of a new triangle, we see that the contribution if $\widehat{c}_{t', -v - 1}$ were equal to $1$ is given by
\[
\binomd{t - t'}{\frac{t-i}{2} - \frac{t'+v+1}{2}}.
\]

Subtracting the overcounted contributions of $i = -v - 1$ over all $t' \in \ZZ$, we have
\begin{align}\label{cti1}
c_{t,i} = \binomd{t}{\frac{t-i}{2}} - \widehat{c}_{1,-v-1} \binomd{t-1}{\frac{t-i}{2} - \frac{2+v}{2}} - \widehat{c}_{3,-v-1} \binomd{t-3}{\frac{t-i}{2} - \frac{4+v}{2}} - \dots 
\end{align}
if $-v - 1$ is odd, and 
\begin{align}\label{cti2}
c_{t,i} = \binomd{t}{\frac{t-i}{2}} - \widehat{c}_{2,-v-1} \binomd{t-2}{\frac{t-i}{2} - \frac{3+v}{2}} - \widehat{c}_{4,-v-1} \binomd{t - 4}{\frac{t-i}{2} - \frac{5+v}{2}} - \dots
\end{align}
if $-v - 1$ is even. Note that both sums terminate after a finite number of terms, since the bottom entry in the binomial coefficient eventually becomes negative. Hence, it suffices to consider only $t' \leq t$.
\end{proof}

In view of Lemma \ref{recursionchat}, to calculate $c_{p,t,i}(n)$, we first compute $\widehat{c}_{p, t',-v - 1}(n)$. 

\begin{lem}\label{chat}
For $t' \ge 1$, we have
\[ 
\widehat{c}_{p, t', -\nu_p(n) - 1}(n) = \sum_{m, \ell_1\in \ZZ^+} \sum_{\substack{1 \leq \ell_2 \leq \cdots \leq \ell_m\\ 2\ell_1+ \dots + 2 \ell_m = t'-\nu_p(n) +1}} (-1)^{m-1} \binomd{2 \ell_1 + \nu_p(n) - 1}{ \ell_1 + \left\lceil \frac{\nu_p(n)}{2} \right\rceil} \binomd{2 \ell_2}{\ell_2} \cdots \binomd{2 \ell_m}{\ell_m}.
\]
\end{lem}

\begin{proof}
Assume the same notation as in the proof of Lemma \ref{recursionchat}.
For each $t'\leq t$, we compute $\widehat{c}_{t', -v - 1}$ recursively using Lemma \ref{recursionchat}. In the case where $v$ is even, then $-v-1$ is odd, so we have
\begin{align*}
\widehat{c}_{1, -v - 1} &= \binomd{1}{\frac{2 + v}{2}}, \\
\widehat{c}_{3, -v - 1} &= \binomd{3}{\frac{4 + v}{2}} - \widehat{c}_{1, -v - 1} \binomd{2}{\frac{4 + v}{2} - \frac{2 + v}{2}} = \binomd{3}{\frac{4 + v}{2}} - \binomd{1}{\frac{2 + v}{2}} \binomd{2}{1},
\end{align*}
\begin{align*}
\widehat{c}_{5, -v - 1} &= \binomd{5}{\frac{6 + v}{2}} - \widehat{c}_{1, -v - 1} \binomd{4}{\frac{6 + v}{2} - \frac{2 + v}{2}} - \widehat{c}_{3, -v - 1} \binomd{2}{\frac{6 + v}{2} - \frac{4 + v}{2}} \\
&=\binomd{5}{\frac{6 + v}{2}} - \binomd{1}{\frac{2 + v}{2}} \binomd{4}{2} - \left( \binomd{3}{\frac{4 + v}{2}} - \binomd{1}{\frac{2 + v}{2}} \binomd{2}{1} \right) \binomd{2}{1}.
\end{align*} 
In general, we have
\[
    \widehat{c}_{t', -v - 1} = \sum_{m, \ell_1\in \ZZ^+}  \sum_{\substack{1 \leq \ell_2 \leq \cdots \leq \ell_m\\ 2\ell_1  +\dots + 2 \ell_m = t'+1}} (-1)^{m-1} \binomd{2 \ell_1 - 1}{\frac{2 \ell_1 + v}{2}} \binomd{2 \ell_2}{\ell_2} \cdots \binomd{2 \ell_m}{\ell_m}.
\]
For the first binomial factor to be nonzero, we must have $\frac{2 \ell_1 + v}{2}\le 2 \ell_1 - 1$, so we may discard the summands where $\ell_1 < 1+ v/2$. Thus, we have 
\[ 
\widehat{c}_{t', -v - 1} = \sum_{m, \ell_1 \in \ZZ^+}  \sum_{\substack{1 \leq \ell_2 \leq \cdots \leq \ell_m\\ 2\ell_1  +\dots + 2 \ell_m = t'+1-v}} (-1)^{m-1} \binomd{2 \ell_1 + v - 1}{\frac{2 \ell_1 + v}{2}} \binomd{2 \ell_2}{\ell_2} \cdots \binomd{2 \ell_m}{\ell_m}.
\]
The situation where $v$ is odd is exactly analogous; in this case, we have
\[ 
\widehat{c}_{t', -v - 1} = \sum_{m, \ell_1\in \ZZ^+}  \sum_{\substack{1 \leq \ell_2 \leq \cdots \leq \ell_m\\ 2\ell_1  +\dots + 2 \ell_m = t'+1-v}} (-1)^{m-1} \binomd{2 \ell_1 + v - 1}{\frac{2 \ell_1 + v+1 }{2}} \binomd{2 \ell_2}{\ell_2} \cdots \binomd{2 \ell_m}{\ell_m}.
\]
Together, these two equalities prove the lemma.
\end{proof}

The following proposition provides the promised criterion for the parity of $c_{p, t, i}(n)$.

\begin{prop}\label{coeff}
For $t \in \ZZ^+$, we have
\begin{align*} 
c_{p, t, i}(n) = \binomd{t}{\frac{t-i}{2}} +\sum_{m,\ell_1 \in \ZZ^+}  \sum_{\substack{1 \leq \ell_2 \leq \cdots \leq \ell_m\\ \ell_1 + \dots + \ell_m \leq \frac{t-1}{2}}} (-1)^m &\binomd{2 \ell_1 + \nu_p(n) - 1}{\ell_1 + \left\lceil \frac{\nu_p(n)}{2} \right\rceil} \binomd{2 \ell_2}{\ell_2} \cdots \binomd{2 \ell_m}{\ell_m} \\
&\binomd{t - 2 L_m + 1 - \nu_p(n)}{\frac{t-i}{2} - L_m - \left\lceil \frac{\nu_p(n)}{2} \right\rceil},
\end{align*}
where $L_m\coloneqq \ell_1 + \dots + \ell_m$.
In particular, we have
\[
c_{p, t, i}(n) \equiv \binomd{t}{\frac{t-i}{2}} - \sum_{1 \leq \ell_1 \leq \frac{t-1}{2}} \binomd{2 \ell_1 + \nu_p(n) - 1}{\ell_1 + \left\lceil \frac{\nu_p(n)}{2} \right\rceil} \binomd{t - 2 \ell_1 + 1 - \nu_p(n)}{\frac{t-i}{2} - \ell_1 - \left\lceil \frac{\nu_p(n)}{2} \right\rceil} \pmod{2}.
\] 
\end{prop}

\begin{proof} 
Once again, assume the notations established above. Lemma \ref{recursionchat} states that 
\begin{align}\label{ctieven}
c_{t,i} = \binomd{t}{\frac{t-i}{2}} - \sum_{\ell \geq 1} \widehat{c}_{2 \ell - 1, -v - 1} \binomd{t - 2 \ell + 1}{\frac{t-i}{2} - \frac{2 \ell_1 + v}{2}}
\end{align}
for $v$ even and
\begin{align}\label{ctiodd} 
c_{t,i} = \binomd{t}{\frac{t-i}{2}} - \sum_{\ell \geq 1} \widehat{c}_{2 \ell, -v - 1} \binomd{t - 2 \ell}{\frac{t-i}{2} - \frac{2 \ell_1 + v + 1}{2}}
\end{align}
for $v$ odd. Substituting equations \eqref{ctieven} and \eqref{ctiodd} into Lemma \ref{chat} yields
\[ 
c_{t,i} = \binomd{t}{\frac{t-i}{2}} + \sum_{t' = 0}^t \quad \smashoperator{\sum_{m, \ell_1 \in \ZZ^+}} \quad  \sum_{\mathbf{x}} (-1)^{m} \binomd{2 \ell_1 + v - 1}{ \ell_1 + \left\lceil v/2 \right\rceil} \binomd{2 \ell_2}{\ell_2} \cdots \binomd{2 \ell_m}{\ell_m} \binomd{t - t' - v}{\frac{t-i}{2} - L_m- \left\lceil v/2 \right\rceil},
\]
where the third sum is taken over all $(m-1)$-tuples $\mathbf{x} = (\ell_2, \dots, \ell_m)$ such that $1 \leq \ell_2 \leq \cdots \leq \ell_m$ and $(2\ell_1 +v - 1) + 2 \ell_2 \dots + 2 \ell_m = t'$. Since $t'$ takes the form of $2 \ell - 1$ or $2 \ell$ depending on the parity of $v$, this proves the first half of the proposition. 
The second half of the proposition now follows from the fact that, for any positive integer $i$, the binomial coefficient $\displaystyle\binomd{2i}{i} = 2^i\cdot \frac{1 \cdot 3 \cdot 5 \cdots (2i-1)}{i!}$ is even.
\end{proof}

We now focus on the special case $p =5$ and $t = 2^{s-1}-1$ and determine the parity of $c_{p,t,i}$ in order to refine the sets of partition values that appear in Proposition \ref{easyset}.

\begin{lem}\label{1stcoeff}
For any positive integer $n$ such that $5 \nmid n$, we have
\[ c_{5, 2^{s-1}-1, i}(n) \equiv 
\begin{cases*}
0 \pmod{2} \ & \textrm{if} $i = 2^{s-1} - 2^m - 1$ for some $m \in \ZZ$ \\
1 \pmod{2} \ & \textrm{if} $i \neq 2^{s-1} - 2^m - 1$ for all $m \in \ZZ$.
\end{cases*}
\]
\end{lem}

\begin{proof}
Let $n$ be a positive integer with $5 \nmid n$.  Setting $p = 5, t = 2^{s-1} - 1$ in Proposition \ref{coeff} and writing $c_i \coloneqq c_{p, t, i}(n)$, we have
\[ 
c_i \equiv \binomd{2^{s-1} - 1}{(2^{s-1} - 1 - i)/2} - \sum_{1 \leq \ell_1 \leq {2^{s-2}-1}} \binomd{2 \ell_1 - 1}{\ell_1} \binomd{(2^{s-1} - 1) - (2 \ell_1 - 1)}{\frac{2^{s-1} - 1 - i}{2} - \ell_1} \pmod{2}.
\] 
For any positive integer $a$, the binomial coefficient $\binom{2^{s-1} - 1}{a}$ is odd for all $0 \leq a \leq 2^{s-1} - 1$. In addition, $\binom{2 \ell_1 - 1}{\ell_1}$ is odd if and only if $\ell_1$ is a power of $2$. Thus, to determine the parity of $c_i$, we only need to analyze the parity of the expression 
\[
c_i \equiv 1 + \sum_{1 \leq r \leq s-3} \binomd{(2^{s-1} - 1) - (2^{r+1} - 1)}{j - 2^r} \equiv 1 + \sum_{1 \leq r \leq s-3} \binomd{2^{s-1} - 2^{r+1}}{j - 2^r} \pmod{2},
\]
where $j \coloneqq \frac{2^s - 1 - i}{2}$. For any $r\le s-3$ and $j\ne 2^r$, we have $\nu_2(2^{s-1} - 2^{r+1}) = r+1$ and $\nu_2(j - 2^r) \leq r$, so $\nu_2(2^{s-1} - 2^{r+1})> \nu_2(j-2^r)$. As a direct application of Lucas's theorem \cite{Lucas}, we see that $\binom{2^{s-1} - 2^{r+1}}{j - 2^r}$ is even, which means that $c_i \equiv 1 \pmod{2}$ when $j$ is not a power of $2$. Otherwise, suppose that $j = 2^m$, in which case $\binom{2^{s-1} - 2^{r+1}}{0} = 1$ is the only odd binomial coefficient, so $c_i \equiv 0 \pmod{2}$. Since $j$ is a power of $2$ if and only if $i = 2^{s-1} - 2^r - 1$ for some integer $r$, this proves the statement.
\end{proof}

We are now ready to prove Theorem \ref{ns}.
\begin{proof}[Proof of Theorem \ref{ns}]

Recall from Lemma \ref{lem1} that
\[ 
T_5^{2^{s-1} - 1} \mid \Delta^{\frac{4^s-1}{3}} \equiv \Delta \pmod{2}. 
\] 
Let $\ell$ be a positive odd integer with $5 \nmid \ell$, and let $m=\frac{4^s - 1}{3}$. 
Since the odd coefficients of $\Delta$ are supported on odd-square exponents, Lemma $\ref{1stcoeff}$ implies that
\[
\sum_{i} a_{\Delta^m}(5^i \ell^2) \equiv 1 \pmod{2},
\]
where the sum is taken over all positive odd integers $i\leq 2^{s-1} - 1$ which cannot be written in the form $2^{s-1} - 2^r - 1$ for any positive integer $r$. Note that we need only consider positive $i$ because $a(5^i \ell^2) = 0$ for all $i < 0$.

Therefore, $a(5^i \ell^2)$ must be odd for at least one of such $i$. It follows from Lemma \ref{lem1} that
\[
p \left( \frac{5^i \ell^2 - \frac{4^s - 1}{3}}{8} - 4^s j \right) \equiv 1 \pmod{2}
\]
is odd for some pentagonal number $j$.
\end{proof}

\nocite{*}
\bibliographystyle{plain}
\bibliography{bib.bib}

\end{document}